\documentclass[12pt]{amsart}
\usepackage{amssymb,latexsym}
\usepackage{enumerate}

\makeatletter
\@namedef{subjclassname@2010}{%
  \textup{2010} Mathematics Subject Classification}
\makeatother

\newtheorem{thm}{Theorem}[section]

\newtheorem{lem}[thm]{Lemma}

\newtheorem{defin}[thm]{Definition}
\newtheorem{rem}[thm]{Remark}
\newtheorem{exa}[thm]{Example}

\numberwithin{equation}{section}

\newcommand{\dist}{\mbox{$\,$dist$\,$}}

\begin{document}


\baselineskip=17pt




\title[SRB-like measures]{SRB-like measures for $C^0$ dynamics}

\author[E. Catsigeras]{Eleonora Catsigeras}
\address{Instituto de Matem\'{a}tica y Estad\'{\i}stica Prof. Ing. Rafael Laguardia \\ Facultad de Ingenier\'{\i}a\\ Av. Herrera y Reissig 565, C.P.11300,
Montevideo, Uruguay}
\email{eleonora@fing.edu.uy}

\author[H. Enrich]{Heber Enrich}

\date{}




\subjclass[2010]{Primary 37A05; Secondary 28D05}

\keywords{Observable measures, SRB measures, physical measures.}


\begin{abstract}
For any continuous map $f\colon M \to M$ on a compact manifold $M$,
we define the  SRB-like    (or   observable)  probabilities   as a generalization of   Sinai-Ruelle-Bowen (i.e. physical) measures. We prove that  $f$  has observable measures, even if SRB measures do not exist. We prove that the definition of observability is optimal, provided that the purpose of the researcher is to describe the asymptotic statistics for Lebesgue almost every initial state. Precisely, the never empty set ${\mathcal O}$ of all the observable measures, is the minimal weak$^*$ compact set of Borel probabilities in $M$ that contains the limits (in the weak$^*$ topology)  of all the convergent subsequences of the empiric probabilities  $ \{(1/n) \sum_{j= 0}^{n-1}\delta_{f^j(x)}\}_{n \geq 1}$, for Lebesgue almost all $x \in M$. We prove that any isolated measure in ${\mathcal O}$ is   SRB. Finally we conclude that if ${\mathcal O}$ is finite or countable infinite, then  there exist  (up to countable many)  SRB measures such that the union of their basins cover $M$ Lebesgue a.e.

\end{abstract}


\maketitle


\section{Introduction}

Let $f\colon M \to M$ be a continuous map in a compact,
finite-dimensional manifold $M$. Let $m$ be a Lebesgue measure
normalized so that $m(M) =1$, and not necessarily $f$-invariant.
We denote $\mathcal P$ the set of all Borel probability measures in
$M$, provided with the weak$^*$ topology, and a metric structure
inducing this topology.

For any point $x \in M $ we denote
$p\omega (x)$ to the set of all the Borel probabilities in $M $ that
are the  limits in the weak$^*$ topology of the convergent subsequences of the following
sequence
\begin{equation}\label{ec1}
\left\{ \frac{1}{n} \; \sum_{j=0}^{n-1} \delta _{f^j(x)}
\right\} _{ n \in \mathbb{N}}
 \end{equation}
 where $\delta_y$ is the Dirac delta probability measure
supported in $y \in M$.
We call the probabilities of the sequence (\ref{ec1}) \em empiric probabilities \em of the orbit of $x$.
We call $p\omega(x)$
 the \em  limit set \em in ${\mathcal P}$ corresponding to  $x \in M$.

It is classic in Ergodic Theory  the following definition:

\begin{defin}
\label{Defphysical} \em
A probability measure $\mu$ is \em physical \em or \em SRB  \em (Sinai-Ruelle-Bowen), if
$\{\mu\} = pw(x)$  for a set $A(\mu)$ of points $x \in M$ that has positive Lebesgue measure. The set $A(\mu)$ is called \em basin of attraction \em of $\mu$.
\end{defin}
In this paper, as in \cite{viana} and Chapter 11 of \cite{BonattiDiazViana}, we agree to name such a probability $\mu$  an SRB  measure (and also physical as in \cite{young}). This preference is based in three  reasons, which are also our motivations:

\vspace{0.1cm}

\noindent  1.    Our scenario includes \em all the continuous systems\em. Most (namely $C^0$ generic) continuous $f$ are not differentiable. So, no Lyapunov exponents necessarily exist, to be able to assume some kind of hyperbolicity. Thus, we can not assume the existence of an unstable foliation with differentiable leaves. Therefore, we aim to study those systems for which the SRB measures    usually defined in the literature (related with an unstable foliation ${\mathcal F}$), do not exist.  We recall a popularly required property for  $\mu$:    the conditional measures $\mu_x$ of $\mu$, along the local leaves ${\mathcal F}_x$ of a  hyperbolic unstable foliation ${\mathcal F}$, are absolute continuous respect to the internal Lebesgue measures of those leaves. But this latter assumption needs the existence of such a regular foliation ${\mathcal F}$. It is well known that the ergodic theory based on this absolute continuity condition    does not work for generic   $C^1$ systems (that are not $C^{1+\alpha}$), see  \cite{RobinsonYoung, Bruin, Avila}. So, it does not work  for most $C^0$-systems.

\vspace{.1cm}

\noindent 2.   In the modern Differentiable Ergodic Theory, for $C^{1 + \alpha}$-systems that have some hyperbolic behavior, one of the ultimate purposes of searching  measures with absolute continuity properties respect to Lebesgue is to find   probabilities that satisfy Definition \ref{Defphysical}. Therefore, if the system is not $C^{1 +\alpha}$, or is not hyperbolic-like, but nevertheless exists some   probability $\mu$    describing the asymptotic behavior of the sequence (\ref{ec1})   for a Lebesgue-positive set of initial states  (i.e. $\mu$ satisfies Definition \ref{Defphysical}), then  one of the  initial purposes of research of Sinai, Ruelle and Bowen in  \cite{bowen2, bowen3, ruelle, sinai}, is also achieved. Therefore, it makes sense   (principally for $C^0$-systems) to call  $\mu$ an SRB measure, if it   satisfies Definition \ref{Defphysical}.

\vspace{.1cm}

\noindent 3.    The SRB-like property of some invariant measures which describe   (modulus   $\varepsilon$ for all $\varepsilon>0$) the   behavior of the sequence (1) for   $n$ large enough and for a Lebesgue-positive set of initial states can be also achieved considering the \em observable measures \em that we introduce in Definition \ref{definicionobservable}, instead of restricting to those in Definition \ref{Defphysical}. This new setting will describe the statistics defined by the sequence (\ref{ec1}) of empiric probabilities for Lebesgue almost all initial state  (see Theorem \ref{toeremaminimal0}). This is particularly interesting in the cases in which SRB-measures do not exist (for instance \cite{keller2} and some of the examples in Section \ref{seccionejemplos} of this paper.) So, in the sequel, we use the words physical and SRB as synonymous, and we apply them only to the probability measures   that satisfy Definition \ref{Defphysical}. To generalize this notion, we will call \em observable or SRB-like or physical-like, \em to those measures   introduced in Definition \ref{definicionobservable}. After this agreement   all SRB  measure are SRB-like but not conversely (we provide Examples in Section \ref{seccionejemplos}).

One of the major problems of the Ergodic Theory of Dynamical Systems, is to find   SRB measures. They   are widely studied
 occupying a relevant
interest for those systems that are $C^{1 + \alpha}$ and show some kind of hyperbolicity
(\cite{pesinsinai}, \cite{pughshub2},
\cite{viana}, \cite{BonattiDiazViana}). One of the reasons for searching those measures, is that they    describe the asymptotic behavior of the sequence (\ref{ec1}) for a Lebesgue-positive set of initial states, namely, for a set of spatial conditions that is not negligible from the viewpoint of  an  observer. One observes, through the SRB measures,     the statistics of the orbits through experiments that measure the time-mean of the future evolution of the system, with Lebesgue almost all initial states. But it is unknown if most differentiable systems exhibit SRB measures  (\cite{palis}). Many interesting  $C^0$-systems   do not exhibit SRB    measures. In fact,  there is evidence that  for many $C^0$ systems, Lebesgue almost all initial states define   sequences   (\ref{ec1}) of empiric probabilities  that are convergent \cite{abdenurandersson}, but none of the measures $\mu$ in such  limits has a Lebesgue-positive basin of attraction $A(\mu)$ as required in Definition \ref{Defphysical} to be an SRB measure \cite{abdenurandersson}.
In \cite{keller}, Keller considers an SRB-like property of a measure, even if the sequence (\ref{ec1}) is not convergent. In fact, he takes   those measures $\mu$ that belong to the set $pw(x)$ for a Lebesgue-positive set of initial states $x \in M$, regardless if $ pw(x)$ coincides or not with $\{\mu\}$. Precisely, Keller considers  those measures $\mu$ for which $\mbox{dist}(\mu, pw(x)) = 0$ for a Lebesgue positive set of points $x \in M$. But, as he also remarks in his definition, that kind of weak-SRB   measures may not exist.  We introduce now the following  notion, which generalizes the notion of observability of Keller, and the notion of SRB measures in Definition \ref{Defphysical}:

\begin{defin}
  \label{definicionobservable} \em
A probability measure $\mu \in \mathcal P$ is \em observable \em  or \em SRB-like \em  or \em physical-like \em if
for all  $\varepsilon
>0 $  the set $A_{\varepsilon }(\mu)= \{x \in M:
 \mbox{ dist } ( p\omega(x), \mu )<  \varepsilon \}$ has
 positive Lebesgue measure.
 The set $A_{\varepsilon}(\mu )$ is called basin of $\varepsilon$-attraction of $\mu$.
 We denote with ${\mathcal O}$ the set of all observable measures.
\end{defin}

It is immediate from Definitions \ref{Defphysical} and \ref{definicionobservable}, that every SRB measure is observable. But not
every observable measure is SRB (we provide examples in Section \ref{seccionejemplos}). It is standard to check that any observable measure is $f$-invariant. (In fact, if ${\mathcal P}_f \subset {\mathcal P}$ denotes the weak$^*$-compact set of $f$-invariant probabilities, since $p\omega(x) \in {\mathcal P}_f$ for all $x$, we conclude that $\mu \in \overline {{\mathcal P}_f} = {\mathcal P}_f$ for all $\mu \in {\mathcal O}$.)
For the experimenter, the observable measures as defined in \ref{definicionobservable} should have  the same relevance  as the SRB measures defined in \ref{Defphysical}.
In fact, the basin of $\varepsilon$-attraction $A_{\varepsilon}(\mu)$ has positive Lebesgue measure \em for all $\varepsilon >0$. \em The $\varepsilon$-approximation lays in the space ${\mathcal P}$ of probabilities, but it can be easily translated (through the functional operator induced by   the probability $\mu$  in the   space   $ C^0(M, \mathbb{R})$) to an  $\varepsilon$-approximation (in time-mean) towards an \lq\lq attractor\rq\rq \ in the ambient manifold $M$. Precisely, if $\mu $ is observable and $x \in A_{\epsilon}(\mu)$ then,   with a frequency that is asymptotically bounded away from zero,    the  iterates $f^{n}(x)$  for certain values of $n $ large enough,  will $\varepsilon$-approach    the   support of $\mu$.   Note that also for an SRB measure $\mu$ this  $\varepsilon$-approximation to the support of $\mu$ holds in the ambient manifold $M$  with $\varepsilon \neq 0$. Namely, assuming that there exists an SRB measure $\mu$, the empiric probability (defined in (\ref{ec1}) for Lebesgue almost all orbit in the basin of $\mu$) approximates, but in general   differs from $\mu$,   after any     time $n \geq 1$ of experimentation which is as large as wanted   but  finite. If the experimenter  aims to observe the orbits during a time $n$   large enough, but always finite,   Definition \ref{definicionobservable} of observability   ensures him a $2 \varepsilon$-approximation to the \lq\lq attractor\rq\rq,  for any given $\varepsilon >0$, while Definition \ref{Defphysical} of physical measures ensures him an $\varepsilon$-approximation.  As none of them guarantees a null error, and both of them guarantee an error smaller than $\epsilon >0 $   for arbitrarily small values of $\epsilon >0 $  (if the observation time is large enough), the  practical meanings of both definitions are similar.

\vspace{.3cm}

 \begin{center}

 \noindent {    \textsc{Statement of the results } }

\end{center}

\begin{thm} \label{teoremaexistencia0}   { \bf (Existence of
observable measures) }

For every continuous map $f$, the space $\mathcal O$ of all observable
measures for $f$ is nonempty and   weak$^*$-compact.

\end{thm}

We prove this theorem  in Section
\ref{proofsMainTheorems}.
It says that Definition \ref{definicionobservable}
  is weak enough to ensure  the existence of observable measures for any continuous $f$.
  But, if considering the set ${\mathcal P}_f$ of all the invariant measures,  one would   obtain also the existence of   probabilities that  describe  completely the limit set $pw(x)$ for a Lebesgue-positive set of points $x \in M$ (if so, for all points in $M$). Nevertheless, that would be less economic. In fact, along Section \ref{seccionejemplos}, we exhibit paradigmatic systems for which   most invariant measures are not observable. Also we show that observable measures (as well as SRB measures  defined in \ref{Defphysical}) are not necessarily ergodic. The ergodic measures, or a subset of them, may be not suitable respect to a non-invariant Lebesgue measure describing the probabilistic distribution of the initial states in $M$. In fact,  there exist examples (we will provide one   in Section \ref{seccionejemplos}), for which the set of points $x \in M$ such that $p\omega (x)$ is an ergodic probability has zero Lebesgue measure.

  In   Definition \ref{Defphysical},  we called basin of attraction $A(\mu)$ of an SRB-measure  $\mu$ to the set $A(\mu) = \{x \in X: p\omega(x) = \{\mu\}\, \}$. Inspired in that definition we introduce the following:

  \begin{defin}
  \label{definitionBasinOfK} \em
  We call   \em basin of attraction $A({\mathcal K})$ \em of any nonempty weak$^*$ compact subset ${\mathcal K}$ of probabilities, to
\begin{equation} \nonumber
A({\mathcal K}) := \{ x \in M: \ p\omega(x) \subset {\mathcal K}\}.  \end{equation}
  \end{defin}
  We are interested in those sets ${\mathcal K} \subset {\mathcal P}$ having basin $A({\mathcal K})$ with positive Lebesgue measure. We are also interested in not adding unnecessary probabilities to the set ${\mathcal K}$. The following result  states that the optimal choice, under those interests, is a nonempty compact subset of the  observable measures    defined in \ref{definicionobservable}.

\begin{thm} \label{toeremaminimal0}{  \bf (Full optimal attraction of ${\mathcal O}$)}

The set $\mathcal O$ of all observable measures for $f$ is the minimal
weak$^*$ compact subset of $\mathcal P$ whose basin of attraction has
total Lebesgue measure. In other words, ${\mathcal O}$ is minimally weak$^*$ compact containing,
for Lebesgue almost all initial state,
the limits of  the convergent subsequences of  \em (\ref{ec1}). \em
\end{thm}
We prove this theorem in Section \ref{proofsMainTheorems}.
Finally, let us state the relations between the cardinality of ${\mathcal O}$ and the existence
 of SRB measures according with Definition \ref{Defphysical}.

\begin{thm} [Finite set of observable measures]
\label{teoremaOfinito}

$\mathcal O$ is finite if and only if there exist finitely many SRB measures
such that the union of
their basins of attraction cover $M$ Lebesgue a.e. In this case $\mathcal O$ is the set of SRB measures.

\end{thm}

We prove this theorem in Section \ref{SectionProofsTeoremasCountable}.

\begin{thm}  [Countable set of observable measures] \label{teoremaOcountableinfinite}
If $\mathcal O$ is countably infinite, then there exist countably infinitely many SRB measures such that
their basins of attraction cover $M$ Lebesgue a.e. In this case   $\mathcal O$ is the weak$^*$-closure of the set of SRB measures.
\end{thm}

 We prove this theorem in Section \ref{SectionProofsTeoremasCountable}.

\vspace{.3cm}

 For systems preserving the Lebesgue measure the main question is
 their ergodicity, and most results of this work translate, for those
 systems, as equivalent conditions to be ergodic.
 The proof of the following result is  standard after Theorem \ref{toeremaminimal0}:

\begin{rem}
\label{teoremaLebesgue0}   {\bf (Observability and
ergodicity.)}
If $f$ preserves the Lebesgue measure $m$ then the following
assertions are equivalent:

\em 1. \em   $f$ is ergodic respect to $m$.

\em 2. \em  There exists a unique observable measure $\mu$ for
$f$.

\em 3. \em   There exists a unique SRB measure $\nu $ for $f$
attracting Lebesgue a.e.

Moreover, if  the assertions above are satisfied, then $m = \mu =
\nu$

\end{rem}

  The ergodicity
 of most maps that preserve the Lebesgue measure is also an open question. (\cite{pughshub2}, \cite{bonattivianawilk}).
Due to Remark \ref{teoremaLebesgue0} this open question  is equivalent to the \em unique observability. \em

\section{The convex-like property of $p\omega(x)$.}

  For each $x \in M$ we have defined the nonempty compact set $p\omega(x) \subset {\mathcal P}_f$ composed by   the limits of all the convergent subsequences of the empiric probabilities in Equality (\ref{ec1}). For further uses we state the following property for the
$p\omega$-limit sets:

\begin{thm} \label{teoremaconvexo} {\bf (Convex-like
property.)}
For every point $x \in M$:

\noindent \em  1. \em If $\mu, \nu \in p\omega(x)$  then for each real number $ 0
\leq \lambda \leq 1$ there exists a measure $\mu_{\lambda} \in
p\omega (x)$ such that $\dist(\mu_{\lambda}, \mu) = \lambda
\dist(\nu, \mu).$

\noindent \em 2.  \em $p\omega(x)$ either has a single element or is uncountable.

\end{thm}
\begin{proof}

 The statement 2 is an immediate consequence of  1.
To prove  1
it is enough to exhibit, in the case  $\mu \neq \nu$, a
convergent subsequence
 of (\ref{ec1}) whose limit $\mu_{\lambda } $ satisfies 1. It is an easy exercise to observe that the existence of such convergent sequence follows (just taking $\varepsilon = 1/n$) from the following lemma \ref{ass1}. \end{proof}

\begin{lem} \label{ass1} For fixed $x \in M$ and for all $n \geq 1$ denote $\mu_n = \frac{1}{n}  \sum_{j= 0}^{n-1} \delta _{f^j(x)}.$  Assume that there exist two weak$^*$-convergent subsequences $\mu_{m_j} \rightarrow \mu, \; \; \mu_{n_j} \rightarrow \nu.$
Then, for all $\varepsilon >0 $ and all $K >0$ there exists
 a natural number $h = h(\varepsilon, K)>K$ such that \em
  $|\dist (\mu_h, \mu) -\lambda \dist (\nu, \mu)| \leq \varepsilon.$
 \end{lem}

\begin{proof}
First let us choose $m_j$ and then $n_j$ such that
\begin{equation}
\nonumber
m_j >K; \;\; \; \frac{1}{m_j} < \frac{\varepsilon}{4}; \;\; \;  \dist (\mu, \mu_{m_j}) < \frac{\varepsilon}{4}; \; \;\;
  n_j > m_j; \; \;\; \dist (\nu, \mu_{n_j}) < \frac{\varepsilon}{4}.\end{equation}
We will consider  the following distance in $\mathcal P$: $$\dist (\rho, \delta) =
  \sum_{i=1}^{\infty}\frac {1}{2^i} \; \left| \int g_i \, d \rho
  - \int g_i \, d \delta \; \right|$$ for any $\rho, \delta \in {\mathcal P}$,
  where $\{g_i\}_{i \in \mathbb{N}}$ is a countable dense subset of $C^0(M, [0,1])$.
Note from the sequence (\ref{ec1}) that $|\int g \, d \mu _n - \int g \, d \mu
 _{n+1}| \leq (1/n) ||g||$ for all $g \in C(M, [0,1])$ and all $n \geq
 1$. Then in particular for $n = m_j + k$, we
 obtain
 \begin{equation}
\label{ec110} \dist(\mu_{m_j + k}, \mu _{m_j + k + 1})  \leq \frac{1}{m_j} < \frac{\varepsilon }{4} \;
 \;\; \mbox { for all }
 k \geq 0
  \end{equation}
Now let us choose a natural number $0 \leq k \leq n_j - m_j$ such
that
$$
\left|\dist(\mu_{m_j}, \mu _{m_j + k})- \lambda \dist(\mu_{m_j},
\mu _{n_j}) \right|< \varepsilon /4 \;\;\mbox{ for the given }
\lambda \in [0,1]$$ Such $k$ does exist because inequality (\ref{ec110}) is
verified for all $k \geq 0$ and moreover
if $ k= 0 $ then $ \dist(\mu_{m_j}, \mu _{m_j + k})= 0 $
 and if  $ k= n_j -m_j $ then $  \dist(\mu_{m_j}, \mu _{m_j + k})=
\dist(\mu_{m_j}, \mu _{n_j }).$ Now renaming $h = m_j + k$,
 applying the triangular
property and tying together the inequalities above, we deduce:
$$
\left| \dist(\mu_{h}, \mu) - \lambda \dist (\nu, \mu)\right| \leq \left| \dist(\mu_{h}, \mu) -
\dist(\mu_{h}, \mu_{m_j})\right|  $$ $$+
 \left| \dist(\mu_{h}, \mu_{m_j}) - \lambda \dist (\mu_{m_j},
\mu_{n_j})\right|+   \lambda \left| \dist
(\mu_{m_j}, \mu_{n_j}) - \dist (\mu_{m_j}, \nu)\right|$$ $$ + \lambda
\left| \dist (\mu_{m_j}, \nu) - \dist (\mu, \nu)\right| < \varepsilon
 $$ \end{proof}

\section{Proof of Theorems \ref{teoremaexistencia0} and \ref{toeremaminimal0}.} \label{proofsMainTheorems}

From the beginning we have fixed a  metric in the space $\mathcal P$ of all Borel probability measures in $M$, inducing its weak$^*$ topology
structure. We denote as $B_\varepsilon (\mu)$   the open ball in
$\mathcal P$, with such a metric,  centered in $\mu \in \mathcal P$ and with radius $\varepsilon
>0$.

 \vspace{.3cm}

\begin{proof}
 ({\em  of Theorem } \ref{teoremaexistencia0}.)   Let us prove that ${\mathcal O}  $ is compact. The
complement ${\mathcal O}^c$ of $\mathcal O$ in $\mathcal P$ is the set of all
probability measures $\mu$ (not necessarily $f$-invariant) such
that for some $\varepsilon = \varepsilon (\mu) >0$ the set $\{x \in M:
p\omega(x) \cap B_{\varepsilon }(\mu) \neq \emptyset\}$ has zero
Lebesgue measure. Therefore ${\mathcal O}^c$ is open  in $\mathcal P$, and
${\mathcal O}$ is a closed subspace of  $\mathcal P$. As $\mathcal P$ is
compact we deduce that  $\mathcal O$ is compact as wanted.

We now prove that $\mathcal O$ is not empty. By
contradiction, assume that ${\mathcal O}^c = {\mathcal P}$. Then
for every $\mu \in {\mathcal P}$ there exists some $\varepsilon =
\varepsilon (\mu) >0$ such that the set $A= \{x \in M: p\omega(x)
\subset (B_{\varepsilon }(\mu))^c \}$ has total  Lebesgue
probability.
As $\mathcal P$ is compact, let us consider a finite covering of $\mathcal
P$ with such open balls $B_{\varepsilon }(\mu)$, say $B_1, B_2,
\ldots B_k$, and their respective sets $A_1, A_2, \ldots A_k$
defined as above. As $m(A_i) = 1$ for all $i= 1, 2, \ldots, k$ we
have that the intersection $B= \cap_{i=1}^k A_i$ is not empty. By
construction, for all $x \in B$ the $p\omega$-limit of $x$ is
contained in the complement of  $B_i$ for all $i = 1, 2 \ldots,
k$, and so it would not be contained in $\mathcal P$, that is the
contradiction ending the proof.
\end{proof}

\begin{proof} ({\em  of Theorem} \ref{toeremaminimal0}.)
Recall Definition \ref{definitionBasinOfK} of the basin of attraction
$A({\mathcal K})$ of any weak$^*$-compact and nonempty set ${\mathcal K}$ of probabi\-lities.
We must prove the following two assertions:

\noindent $1. \  $ $m(A({\mathcal O})) = 1$, where $m$ is the Lebesgue measure.

\noindent $2. \ $ ${\mathcal O}$ is minimal among all the compact sets ${\mathcal K} \subset {\mathcal P} $ with such a property.

Define the following family $\aleph$ of sets of probabilities:
$$\aleph= \{ {\mathcal K} \subset {\mathcal P} : \, {\mathcal K} \mbox{ is compact  and }  m(A({\mathcal K})) = 1 \}.$$
Therefore $\aleph$ is composed by all the weak$^*$ compact sets ${\mathcal K}$ of probabilities such that
 $p\omega (x) \subset {\mathcal K}$
for Lebesgue almost
every point
$x \in M  $. The family $\aleph$ is not empty since it contains the set ${\mathcal P}_f$ of all the invariant probabilities. So, to prove Theorem \ref{toeremaminimal0}, we must prove that ${\mathcal O} \in \aleph$ and ${\mathcal O} = \bigcap_{{\mathcal K} \in \aleph} {\mathcal K}$.

  \label{pruebaGenErgAttract}

Let us first prove that ${\mathcal O} \subset {\mathcal K}$ for all ${\mathcal K} \in \aleph$. This is equivalent to prove that if ${\mathcal K} \in \aleph$ and if $\mu \not \in {\mathcal K}$ then $\mu \not \in {\mathcal O}$.

If $\mu \not \in {\mathcal K}$ take $\varepsilon = \dist (\mu, {\mathcal
K})>0$.  For all $x \in A({\mathcal K})$ the set $p\omega (x)\subset
{\mathcal K} $ is disjoint from the ball $B_{\varepsilon }(\mu)$. But
almost all Lebesgue point is in $  A({\mathcal K})$, because ${\mathcal K}
\in \aleph$. Therefore $p\omega (x) \cap B_{\varepsilon }(\mu) =
\emptyset$ Lebesgue a.e. This last assertion, combined with  Definition
$\ref{definicionobservable}$ and the compactness of the set $p\omega(x)$
imply that $\mu \not \in {\mathcal O}$, as wanted.

Now let us prove that $m(A({\mathcal O})) =1.$  After Theorem \ref{teoremaexistencia0} the set $\mathcal O$ is compact and nonempty. So,
for any $\mu \not \in {\mathcal O}$ the distance $\dist
(\mu, {\mathcal O})$ is positive.  Observe that the complement  ${\mathcal O}^c$ of
$\mathcal O$ in $\mathcal P$ can be written as the increasing union of
compacts sets ${\mathcal K}_n$ (not in the family $\aleph $) as
follows:\begin{equation}\label{ec3}
{\mathcal O}^c = \bigcup
_{n=1}^{\infty} {\mathcal K}_n, \;\;\;\;\;\;\;\;{\mathcal K}_n = \{\mu \in
{\mathcal P}: \dist (\mu, {\mathcal O}) \geq 1/n \} \; \subset  \; {\mathcal
K}_{n+1}
\end{equation}  Let us consider the   sequence $A'_n= A'({\mathcal K}_n)$ of sets
in $M$, where $A'({\mathcal K})$ is defined as follows:
\begin{equation}
\label{ec111} A'({\mathcal K}): = \{x \in M: p\omega(x)\cap {\mathcal K} \neq
\emptyset \}.\end{equation}
 Denote
$A'_{\infty}= A'({\mathcal O}^c)$. We deduce from (\ref{ec3}) and (\ref{ec111})  that:
$$A'_{\infty } = \bigcup _{n=1}^{\infty} A'_n, \;\;\;\; m(A'_n)\rightarrow m(A'_{\infty})
 = m(A'({\mathcal O}^c)).$$ To end the proof is now enough to
show that $m(A'_n)=0$ for all $n \in \mathbb{N}$.

In fact,  $A'_n= A'({\mathcal K}_n)$ and ${\mathcal K}_n$ is compact and
contained in ${\mathcal O}^c$. By Definition
\ref{definicionobservable}  there
exists a finite covering of ${\mathcal K}_n$ with open balls ${\mathcal
B}_1, {\mathcal B}_2, \ldots, {\mathcal B}_k$ such that
\begin{equation}\label{ec4}
 m(A'({\mathcal B}_i))= 0 \;\;\; \mbox{for all } i = 1, 2, \ldots , k
 \end{equation}
 By (\ref{ec111})  the
finite collection of sets $A'({\mathcal B}_i); \; i= 1,2, \ldots, k$
cover $A'_n$  and therefore (\ref{ec4}) implies $m(A'_n) = 0$ ending the
proof. \end{proof}

\section{Proof of Theorems \ref{teoremaOfinito} and \ref{teoremaOcountableinfinite}} \label{SectionProofsTeoremasCountable}


\begin{lem}
\label{lemmaObservablesAisladas}
If an observable or SRB-like measure $\mu$  is isolated in the set ${\mathcal O}$ of all observable measures, then it is an SRB measure.
\end{lem}

\begin{proof}
Recall that we denote ${\mathcal B}_\varepsilon (\mu)$   the open ball in
$\mathcal P$ centered in $\mu \in \mathcal P$ and with radius $\varepsilon
>0$. Since $\mu $ is isolated in   ${\mathcal O}$, there exists $\varepsilon_0>0$ such that
the set $ \overline{ {\mathcal B}_{\varepsilon_0}(\mu)} \ \setminus \{\mu\}$ is disjoint from $ {\mathcal O} $.  After Definition \ref{definicionobservable}, \
$m(A) >0$, where $A:= A_{\varepsilon_0}(\mu) = \{x \in M: \ \mbox{dist}(p \omega(x), \mu) < \varepsilon_0\} .$

\noindent After Definition \ref{Defphysical}, to prove that $\mu$ is SRB it is enough to prove that for $m$-almost all $x \in A$ the limit set $p \omega(x)$ of the sequence (\ref{ec1}) of empiric probabilities, is $\{\mu\}$. In fact, fix and arbitrary $0 <\varepsilon < \varepsilon_{0}$. The compact set  $\overline{{\mathcal B}_{\varepsilon_0}(\mu)}\setminus {\mathcal B}_{\varepsilon}(\mu)$ is disjoint from ${\mathcal O}$, then it can be covered with a finite number of open balls ${\mathcal B}_1, {\mathcal B}_2, \ldots, {\mathcal B}_k$ such that $m(A_i) = 0$  for all $i = 1, \ldots, k$, where $A_i := \{x \in M: \ p\omega(x) \cap {\mathcal B}_i \neq \emptyset\}$. Thus, for $m$-a.e. $x \in A$ the limit set $p \omega(x)$ intersects $B_{\varepsilon}(\mu)$ but it does not intersect ${{\mathcal B}_{\varepsilon_0}(\mu)}\setminus {\mathcal B}_{\varepsilon}(\mu)$. After Theorem \ref{teoremaconvexo}, we conclude that $p\omega(x) \subset {\mathcal B}_{\varepsilon}(\mu)$ for Lebesgue almost all $x \in A$. Taking the values  $\varepsilon_n= 1/n$, for all $n \geq 1$, we deduce that $p\omega(x) = \{\mu\}$ for $m-$ a.e. $x \in A$, as wanted.
\end{proof}

\begin{proof} ({\em   of Theorem } \ref{teoremaOfinito}.)
Denote ${\mbox{SRB}}$ to the (a priori maybe empty) set of all SRB measures, according with Definition \ref{Defphysical}. It is immediate, after Definition \ref{definicionobservable}, that ${\mbox{SRB}} \subset {\mathcal O}$.
If ${\mathcal O}$ is finite, then all its measures are isolated, and after Lemma \ref{lemmaObservablesAisladas}, they are all $SRB$ measures. Therefore ${\mbox{SRB}} = {\mathcal O}$ is finite. Applying Theorem \ref{toeremaminimal0} which states the full attraction property of ${\mathcal O}$, it is obtained that $m(A({\mbox{SRB}}))= 1$, where $A({\mbox{SRB}}) = \bigcup_{\displaystyle{\mu} \in {\footnotesize {\mbox{SRB}}}} \, A(\mu)$,  being $A(\mu)$ the basin of attraction of the SRB measure $\mu$. Therefore, we conclude that, if ${\mathcal O}$ is finite, there exist a finite number of SRB measures such that the union of their basins cover Lebesgue almost all $x \in M$, as wanted.
Now, let us prove the converse statement. Assume that ${\mbox{SRB}}$ is finite and the union of the basins of attraction of all the measures in ${\mbox{SRB}}$ cover Lebesgue almost all $x \in M$. After the minimality property of   ${\mathcal O} $ stated in Theorem \ref{toeremaminimal0}, ${\mathcal O} \subset {\mbox{SRB}}$. On the other hand, we have ${\mbox{SRB}} \subset {\mathcal O}$. We deduce that ${\mathcal O} = {\mbox{SRB}}$, and thus ${\mathcal O}$ is finite, as wanted. \end{proof}

To prove Theorem \ref{teoremaOcountableinfinite} we need the following  Lemma (which in fact holds in any compact metric space ${\mathcal P}$).

\begin{lem} \label{lemmaObservablesNumerable}
 If the compact subset ${\mathcal O} \subset {\mathcal P}$ is countably infinite, then the subset ${\mathcal S}$ of its isolated points is not empty, countably infinite and $\overline{{\mathcal S}} = {\mathcal O}$.
Therefore, $    \mbox{\em dist}(\nu, {\mathcal O}) = \mbox{\em dist}(\nu, {\mathcal S}) $  for all $\nu \in {\mathcal P}$.
\end{lem}

\begin{proof}
The set ${\mathcal O} \subset {\mathcal P}$ is not empty and compact, after Theorem \ref{teoremaexistencia0}. Assume by contradiction that ${\mathcal S}$ is empty. Then ${\mathcal O}$ is   perfect, i.e. all measure of ${\mathcal O}$ is an accumulation point. The set ${\mathcal P}$ of all the Borel probabilities in $M$ is a Polish space, since it is metric and compact. As nonempty perfect sets in a Polish space always have the cardinality of the continuum \cite{kech}, we deduce that ${\mathcal O}$ can not be countably infinite, contradicting the hypothesis.

  Even more, the argument above also shows that
    if ${\mathcal O}$ is countable infinite, then it does not contain nonempty perfect subsets.

Let us prove now that the subset ${\mathcal S}$ of isolated measures of ${\mathcal O}$ is countably infinite. Assume by contradiction that ${\mathcal S}$ is finite. Then ${\mathcal O} \setminus {\mathcal S}$ is nonempty and compact, and by construction has not isolated points. Therefore it is a nonempty perfect set, contradicting the assertion    proved above.

It is left to prove that $\mbox{dist}(\nu, {\mathcal O}) = \mbox{dist}(\nu, {\mathcal S})$ for all $\nu \in {\mathcal P}.$  This assertion, if proved, implies in particular that $\mbox{dist}(\mu, {\mathcal S}) = 0$ for all $\mu \in {\mathcal O}$, and therefore, recalling that ${\mathcal O}$ is compact, it implies $ \overline{\mathcal S} ={\mathcal O} $.

To prove that $\mbox{dist}(\nu, {\mathcal O}) = \mbox{dist}(\nu, {\mathcal S})$  for all $\nu \in {\mathcal P}$, first fix $\nu$  and take  $\mu \in {\mathcal O}$ such that $\mbox{dist}(\nu, {\mathcal O}) = \mbox{dist}(\nu,\mu) $. Such a probability $\mu$ exists because ${\mathcal O}$ is compact. If $\mu \in {\mathcal S}$, then the equality in the assertion is obtained trivially. If $\mu \in {\mathcal O} \setminus {\mathcal S}$, fix any $\varepsilon >0$ and denote ${\mathcal B}_{\varepsilon}(\mu )$ to the ball of center $\mu$ and radius $\varepsilon$. Take $\mu' \in   {\mathcal S} \bigcap \mathcal {\mathcal B}_{\varepsilon}(\mu)$.   Such $\mu'$ exists because, if not, the nonempty set
${\mathcal B}_{\varepsilon}(\mu) \cap {\mathcal O}$ would be perfect, contradicting the above proved assertion. Therefore,
$\mbox{dist}(\nu, {\mathcal S}) \leq \mbox{dist}(\nu, \mu') \leq \mbox{dist}(\nu, \mu) + \mbox{dist}(\mu, \mu') = \mbox{dist}(\nu, {\mathcal O}) + \mbox{dist}(\mu, \mu'). $ So, $   \mbox{dist}(\nu, {\mathcal S})< \mbox{dist}(\nu, {\mathcal O}) + \varepsilon.$
As this inequality  holds for all $\varepsilon >0$, we conclude that $\mbox{dist}(\nu, {\mathcal S}) \leq \mbox{dist}(\nu, {\mathcal O})$. The opposite inequality is immediate, since ${\mathcal S} \subset {\mathcal O}$.  \end{proof}

  \begin{proof} ({\em  of Theorem \ref{teoremaOcountableinfinite}}.) Denote ${\mathcal S} $ to the set of isolated measures in ${\mathcal O}$. After Lemma \ref{lemmaObservablesNumerable}, ${\mathcal S} $ is countably infinite. Thus, applying Lemma \ref{lemmaObservablesAisladas}, $\mu $ is   SRB   for all $\mu \in {\mathcal{S}}$. Then, there exist countably infinitely  many SRB measures (those in ${\mathcal S}$ and possibly some others    in ${\mathcal O} \setminus {\mathcal S}$). Denote $ \mbox{SRB}  $ to the set of all SRB measures.  After Lemma \ref{lemmaObservablesNumerable} ${\mathcal O} = \overline {\mathcal S} \subset \overline {\mbox{SRB}} \subset {\mathcal O}$. So $\overline {\mbox{SRB}} = {\mathcal O}$. It is only left to prove   that the union of the basins of attractions $A(\mu_i)$, for all $\mu_i \in \mbox{SRB}$ covers Lebesgue almost all $M$.  Denote $m$ to the Lebesgue measure.
 Applying   Theorem \ref{toeremaminimal0}:
 $p\omega(x) \subset {\mathcal O} \ \ \ m-\mbox{a.e. } x \in M.$  Together with Theorem \ref{teoremaconvexo} and with the hypothesis of countability of ${\mathcal O}$, this last assertion implies that for $m-$ a.e. $x \in M$ the set $p \omega (x)$ has a unique element $ \{\mu_x\} \subset {\mathcal O}$. Then:  \begin{equation}
 \label{eqn66}
  \ \ \ \ \ \ \ \ \ \ \ \ p\omega(x) = \{\mu_x\}  \subset {\mathcal O} \ \ \ m-\mbox{a.e. } x \in M. \ \ \ \ \ \ \ \      \end{equation}

 \noindent We write ${\mathcal O} = \{\mu_i: \ i= 1, \ldots, n\}$, where $\mu_i \neq \mu_j$ if $i \neq j$. Denote $A = \bigcup_{i \in \mathbb{N}} A(\mu_i)$, where  $A(\mu_i) := \{x \in M: \ \mu_x = \mu_i\}$. Assertion    (\ref{eqn66})
   can be written as $m(A)= 1$. In addition, $A(\mu_i) \cap A(\mu_j) = \emptyset $ if $\mu_i \neq \mu_j$. So $1 = \sum_{i= 1}^{+ \infty} m (A(\mu_i))$. After Definition \ref{Defphysical}: $\mbox{SRB} = \{\mu_i \in {\mathcal O}: \; m (A(\mu_i)) >0 \}$. We conclude that $  \sum_{{\displaystyle{\mu_i}} \in {\footnotesize{\mbox{SRB}}}} m(A(\mu_i)) = \sum_{i= 1}^{+ \infty} m (A(\mu_i)) = 1 $, as wanted. \end{proof}

\section{Examples} \label{seccionejemplos}

\begin{exa} \em
For any
transitive $C^{1+\alpha}$ Anosov diffeomorphism the unique SRB measure $\mu$ is the
unique observable measure. But there are also infinitely many other
ergodic and non ergodic invariant probabilities, that are not
observable (for instance those supported on the periodic orbits).
\end{exa}

\begin{exa}  \label{ejemploHu} \em

 In \cite{huyoung} it is studied the class of diffeomorphisms $f$ in
 the two-torus obtained from an Anosov when the unstable
 eigenvalue of $df$ at a fixed point $x_0$ is weakened to be 1, maintaining
 its stable eigenvalue strictly smaller than 1, and the uniform
hyperbolicity
 outside a neighborhood of $x_0$. It is proved that $f$ has a single
 SRB measure, which is the Dirac delta $\delta_{x_0}$ supported on $x_0$,
  and that its basin
 has total Lebesgue measure. Therefore, this is the single
 observable measure for $f$, it is ergodic and   there are infinitely many other
 ergodic and non ergodic invariant measures that are not observable.

\end{exa}

\begin{exa}  \label{ejemploCao} \em
The diffeomorphism $f\colon  [0,1]^2 \to [0,1]^2; \;\; f(x,y) =
(x/2, y)$ has ${\mathcal O}$ as the set
of  Dirac delta measures $\delta _{(0,y)}$ for all  $ y \in
[0,1]$. In this case $\mathcal O$ coincides with the set of all
ergodic invariant measures for $f$.
 Note that, for instance, the one-dimensional Lebesgue measure on
the interval $[0]\times [0,1]$ is invariant and  not observable, and that there
are not SRB-measures as defined in \ref{Defphysical}.
This example also shows that the set ${\mathcal O}$ of observable measures is not necessarily
 closed on convex combinations.
\end{exa}

\begin{exa} \em
 The maps exhibiting infinitely many simultaneous hyperbolic sinks $\{x_i\}_{i \in \mathbb{N}}$, constructed from
 Newhouse's theorem (\cite{newhouse}) has a  space ${\mathcal O}$ of observable measures which contains $\delta_{x_i}$ for all $i \in \mathbb{N}$, which, moreover, are physical measures and isolated in ${\mathcal O}$.  Also the maps exhibiting infinitely many H\'{e}non-like
 attractors, constructed by Colli in \cite{colli}, has a space of
 observable measures that contains countably infinitely many isolated
 probabilities (the SRB measures supported on the H\'{e}non-like
 attractors).
\end{exa}

\begin{exa}  \label{ejemplodossillas} \em
The following example (attributed to Bowen  \cite{takens, golenish} and early cited in \cite{japon})  shows that  even if the system is so regular as $C^{2}$, the space of
observable measures may be formed by the
limit set of   the non convergent sequence  (\ref{ec1})  for Lebesgue almost all
initial states. Consider a diffeomorphism $f$ in a ball of $\mathbb{R}^2$ with two
hyperbolic saddle points $A$ and $B$  such that a half-branch of the unstable
global manifold $W_{half}^u(A)\setminus\{A\}$  is an embedded arc that coincides
 with a half-branch of the stable global manifold $W_{half}^s(B)\setminus\{B\}$,  and conversely
   $W_{half}^u(B)\setminus\{B\} =  W_{half}^s(A) \setminus\{A\} $.
  Take $f$ such that there exists a source $C \in U$ where $U$ is the topological
open  ball   with boundary $W_{half}^u(A)\cup W_{half}^u(B) $. One can design $f$ such that for all $x \in U$
the $\alpha$-limit  is $\{C\}$ and the $\omega $-limit contains $\{A, B\}$.
  If the eigenvalues of the derivative of
$f$ at $A$ and $B$ are adequately chosen as specified in \cite{takens, golenish}, then  the empiric
  sequence  (\ref{ec1}) for all $x \in U \setminus \{C\}$ is not
convergent. It has at least two   subsequences convergent to different convex combinations of the Dirac
deltas $\delta_A$ and
  $\delta _B$.
 Applying Theorem \ref{teoremaOcountableinfinite} there exist uncountably many   observable measures. In addition, as observable measures are invariant
under $f$, due to Poincar\'{e} Recurrence Theorem  all of them are supported on $\{A\} \cup \{B\}$. So,
after Theorem \ref{teoremaconvexo}
all the observable measures are  convex combinations of $\delta _A$ and $\delta _B$ and form a segment
in the space ${\mathcal M}$ of probabilities.
 This example   shows that the observable measures are not
necessarily ergodic.

 Finally, the eigenvalues of $df$ at the saddles $A$  and $B$ can be adequately modified to obtain, instead of the result above, the convergence of the sequence (\ref{ec1})  as stated in Lemma (i) of page 457 in \cite{japon}. In fact, taking conservative saddles (and  $C^0$ perturbing $f$     outside small neighborhoods of the saddles $A$ and $B$  so the topological $\omega$-limit of the orbits  in $U \setminus \{C\}$   still contains $A$ and $B$), one can get  for all $x \in U \setminus \{C\}$  a sequence (\ref{ec1})   that is  convergent to a single measure $\mu = (\lambda) \delta_A+ (1 - \lambda) \delta _B$, with a fixed constant $0 <\lambda < 1$. So  $\mu$ is physical according with Definition \ref{Defphysical}, and is the unique observable measure. This proves that physical measures are not necessarily ergodic.
\end{exa}

\begin{exa} \label{ejemploBDV}
 \em  Consider a partially hyperbolic $C^{2 }$ diffeomorphism $f$, as defined in Section 11.2 of \cite{BonattiDiazViana}. In this family of examples, we will assume that for all $x \in M$  there exists a $df$-invariant dominated splitting
$ TM = E^u \oplus E^{cs}$, where
the sub-bundle $E^u$ is uniformly
expanding, has positive constant dimension, and the expanding exponential rate of $df|_{E^u}$ dominates that
 of $df|_{E^{cs}}$. Through every point $x \in   M$
there
exists a unique $C^2$ injectively immersed
 unstable manifold $F^u(x)$ tangent
to $E^u$. We provide below a concrete example for which such an $f$ has not any SRB-measure according with  Definition \ref{Defphysical}. Nevertheless, in Subsection 11.2.3 of \cite{BonattiDiazViana}  it is proved that $f$ possesses   probability measures $\mu$ that are Gibbs u-states; namely, $\mu$ has conditional measures $\mu_x$ respect to the unstable foliation ${\mathcal F}^u$ that are absolutely continuous respect to the internal Lebesgue measures $m^u_x$ along the leaves ${\mathcal F}^u_x$.   Precisely,   Theorem 11.16 of \cite{BonattiDiazViana} states   that  for all $x$ in a  set $E \subset M$ of initial states  such that $m_y^u(  {\mathcal F}^u_y \setminus E) = 0  $ for all $y \in M$, the convergent subsequences of the empiric probabilities (\ref{ec1})  converge  to   Gibbs u-states  (depending, a priori, of the point $x \in E$).  We provide below a concrete example for which the set $E$ has full Lebesgue measure in the ambient manifold $M$. Therefore  in this example, Theorem 11.16 of \cite{BonattiDiazViana} implies that  for Lebesgue almost all $x \in M$  the limit set $p\omega(x)$ of the sequence (\ref{ec1}) is contained in the set of Gibbs u-states. Combining  this result with   Theorem \ref{toeremaminimal0} of this paper, we deduce that in this example all the    observable or SRB-like measures    are  Gibbs u-states. Nevertheless not all Gibbs u-states are necessarily observable, since Gibbs u-states form a convex set but   ${\mathcal O}$   is not necessarily convex. Moreover, after Theorems \ref{teoremaOfinito} and \ref{teoremaOcountableinfinite}, and since in the example below there does not exist any SRB measure, the set ${\mathcal O}$ (and thus also the set of Gibbs u-states)  is    uncountable. Besides, in the example below this fact  holds simultaneously with the property that the sequence (\ref{ec1}) of empiric probabilities converge for Lebesgue almost all initial state. This latter property, and the statement that the observable measures are  Gibbs u-states, are  two remarkable    differences between  this example  \ref{ejemploBDV} and the example \ref{ejemplodossillas}. For both, no SRB measure exists and the set ${\mathcal O}$ is   uncountable.

To illustrate the ideas above, let us consider     (even being a trivial case of partially hyperbolic system)  the   $C^2$ map $f: \mathbb{T}^3 \mapsto \mathbb{T}^3$ in the three-dimensional torus $\mathbb{T}^3= (\mathbb{S}^1)^3$ constructed by
$f(x,y,z) = (x, g(x,y)) $
where $g: \mathbb{T}^2 \mapsto \mathbb{T}^2$ is a (transitive) $C^2$ Anosov.
After Sinai Theorem  there exists $\mu_1$ in the two-torus, which is $g$-ergodic, SRB-measure   and a Gibbs u-state for $g$.
Thus, the sequence (\ref{ec1}) of the empiric probabilities  for Lebesgue almost all initial states $(x,y,z) \in \mathbb{T}^3$, converges to a measure $\mu_{x} = \delta_x \times \mu_1$, which is supported on   a 1-dimensional unstable manifold injectively immersed in the two torus $\{x\} \times \mathbb{T}^2$. For different values of $x \in \mathbb{S}^1$ the measures $\mu_{x}$ are mutually singular, since they are supported on disjoint compact 2-torus embedded in $\mathbb{T}^3$. For each measure $\mu_{x}$ in $\mathbb{T}^3$, the basin of attraction $A({\mu_{x}})$ (as defined in \ref{Defphysical}) has zero Lebesgue measure in the ambient manifold $\mathbb{T}^3$. So, none of the probabilities $\mu_{x}$ is SRB for $f$. Nevertheless, after Theorem \ref{toeremaminimal0}, the set of all these measures $\mu_{x} $ (which is easy to check to be weak$^*$-compact), coincides with the set ${\mathcal O}$ of observable SRB-like measures for $f$. By construction of this concrete example any $\mu \in {\mathcal O}$ is a Gibbs u-state. Besides,   any $\mu \in {\mathcal O}$ is ergodic, and since there exist  many observable probabilities  and since all convex combination of Gibbs u-states is also a Gibbs u-state, we conclude   that there exist  Gibbs u-states that are not observable.

\end{exa}

\subsection*{Acknowledgements}
The first author was partially financed by CSIC of Universidad de la Rep\'{u}blica and ANII of Uruguay. We thank   Yakov  Sinai for his   comments about the very preliminary version of this paper,   Stephano Luzzatto, Martin Andersson and  the   anonymous  referees   for their   suggestions, and Anton Gorodetski   for providing us the bibliographic reference \cite{golenish}.

\vspace{.5cm}

\end{document}